\documentclass[12pt]{article}
\usepackage{amssymb,amsmath,graphicx,amsthm,mathrsfs}
\usepackage{color}
\usepackage[colorlinks=true]{hyperref}
\hypersetup{linkcolor=blue, urlcolor=green, citecolor=red}
\newtheorem{Theorem}{Theorem}[section]
\newtheorem{Lemma}[Theorem]{Lemma}
\newtheorem{Proposition}[Theorem]{Proposition}

\newtheorem{Remark}[Theorem]{Remark}
\numberwithin{equation}{section}

\def\bb{\begin{equation}} \def\ee{\end{equation}}

\title{Two equivalence theorems of different kinds of optimal control problems for  Schr\"{o}dinger equations}
\author{Yubiao Zhang\thanks{
School of Mathematics and Statistics, Wuhan University, Wuhan,
430072, China. Email: yubiao\b{\;}zhang@whu.edu.cn} }

\begin{document}

\date{}
\maketitle
\begin{abstract} This paper builds up  two  equivalence theorems for different kinds of   optimal control problems of internally controlled  Schr\"{o}dinger equations.
 The first one concerns with the equivalence of the minimal norm  and the  minimal time  control problems. (The minimal time control problems are also called the first type of optimal time control problems.) The targets of the aforementioned two kinds of  problems are the  origin of the state space.
 The second one deals with the equivalence of three optimal control problems which are optimal target control problems, optimal norm control problems and the second type of optimal time control problems. These two theorems were estabilished for heat equations in  \cite{2} and \cite{1} respectively.\\

\noindent\textbf{Key words.} optimal controls, minimal time, minimal norm, optimal distance, optimal time, optimal norm, Schr\"{o}dinger equations\\

\noindent\textbf{AMS Subject Classifications.} 35Q93, 93C20
\end{abstract}

\bigskip
\section{Introduction}

 $\;\;\;\;$ We begin with introducing the minimal time and the minimal norm control problems. Let  $\Omega \subseteq\mathbb{R}^n$
be a bounded domain with a  $C^\infty$ boundary $\partial \Omega$. Let  $\omega\subset \Omega$ be an open and nonempty subset  satisfying the geometric control condition provided in  \cite{3}.
Write $\chi_\omega$ for the characteristic function of $\omega$.
Write $L^2(\Omega)$ for the space $L^2(\Omega; \mathbb{C})$.
Consider the following two Schr\"{o}dinger equations:
\begin{eqnarray}\label{1.4}
 \left\{
  \begin{array}{lll}
   \partial_t y+i\Delta y=\chi_{\omega}u  &\mbox{in}&\Omega\times\mathbb R^+,\\
   y=0  &\mbox{on} &\partial \Omega\times \mathbb R^+,\\
   y(0)=y_0  &\mbox{in} &\Omega
  \end{array}
 \right.
\end{eqnarray}
 and
\begin{eqnarray}\label{1.5}
 \left\{
  \begin{array}{lll}
   \partial_t y+i\Delta y=\chi_{\omega}v  &\mbox{in}&\Omega\times (0,T),\\
   y=0  &\mbox{on} &\partial \Omega\times (0,T),\\
   y(0)=y_0  &\mbox{in} &\Omega.
  \end{array}
 \right.
\end{eqnarray}
Here,  $y_0\in L^2(\Omega)$, $T>0$ and controls $u(\cdot)$ and $v(\cdot)$ are taken from $L^{\infty}(\mathbb R^+;L^2(\Omega))$ and $L^{\infty}((0,T);L^2(\Omega))$ respectively. Denote by $y(\cdot;u)$ and $y(\cdot;v)$ the solutions to (\ref{1.4}) and (\ref{1.5}) respectively. Throughout this paper, write $\langle\cdot,\cdot\rangle$ and $\|\cdot\|$ for the usual inner product and the norm in $L^2(\Omega)$ respectively. We would like to mention that when $\omega$ satisfies the geometric control condition, Schr\"{o}dinger equations have the  $L^\infty-$exact controllability in any time interval by combining Theorem 4.4 in \cite{20} and Proposition 2.1 in \cite{4}.

Given $T >0$ and $ M >0$, we define two admissible control sets:
\begin{description}
 \item ~~~$\mathcal U_M\triangleq\{u\in L^{\infty}(\mathbb R^+;L^2(\Omega));~ \|u(t)\| \leq M~a.e.~ t\in\mathbb R^+ \mbox{ and } \exists s>0~s.t.~y(s;u)=0\}$;
 \item ~~~$\mathcal V_T\triangleq\{v\in L^{\infty}((0,T);L^2(\Omega));~ y(T;v)=0\}$.
\end{description}
Since the  Schr\"{o}dinger equation has the $L^\infty$-exact controllability, the set $ \mathcal V_T$ is non-empty; and meanwhile it follows from
Theorem 3.1 in \cite{6} that the set  $\mathcal U_M$ is non-empty. Now, the minimal time and the minimal norm problems associated with $M$ and $T$ respectively are as follows:
\begin{description}
 \item  ~~~$(TOCP)_M$: $T_M\triangleq\inf\{s;~ u\in\mathcal U_M \mbox{ and } y(s;u)=0\}$;
 \item  ~~~$(NOCP)_T$: $M_T\triangleq\inf\{\|v\|_{L^{\infty}((0,T);L^2(\Omega))};~ v\in\mathcal V_T\}$.
\end{description}
The above two problems have the same target $\{0\}\in L^2(\Omega)$. The numbers $T_M$ and $M_T$ are called accordingly the minimal time and the minimal norm to $(TOCP)_M$ and $(NOCP)_T$. A control
$u^*\in \mathcal U_M$ is called an optimal control to  $(TOCP)_M$ if $y(T_M; u^*)=0$, while a control $v^*\in \mathcal V_T$ is called an optimal control to $(NOCP)_T$ if $\|v^*\|_{L^{\infty}((0,T);L^2(\Omega))}=M_T$. For the existence of the optimal controls to both
 $(TOCP)_M$  and  $(NOCP)_T$, see Theorem 3.3 in \cite{6} and Lemma \ref{Lemma zhang} in section 2.

Next, we will introduce other three different optimal control problems. Fix a $T>0$. Consider the following Schr\"{o}dinger equation:
\begin{eqnarray}\label{1.1}
 \left\{
  \begin{array}{lll}
   \partial_t y+i\Delta y=\chi_{\omega}\chi_{(\tau,T)}u  &\mbox{in}&\Omega\times(0,T),\\
   y=0  &\mbox{on} &\partial \Omega\times (0,T),\\
   y(0)=y_0  &\mbox{in} &\Omega,
  \end{array}
 \right.
\end{eqnarray}
where $y_0\in L^2(\Omega)$, $\tau\in[0,T)$, $u\in L^\infty((0,T);L^2(\Omega))$, and  $\chi_{(\tau,T)}$ is the characteristic function of  $(\tau,T)$. Denote by $y(\cdot;\chi_{(\tau,T)}u,y_0)$ the solution to Equation (\ref{1.1}). Let  $z_d\in L^2(\Omega)$ verify
\begin{equation}\label{wang1.4}
r_T\triangleq\|y(T;0,y_0)-z_d\|>0.
\end{equation}
Define the following target sets:
$$B(z_d,r)\triangleq\{z\in L^2(\Omega); \|z-z_d\| \leq r\},\;\; r\geq 0. $$
When $M\geq 0$, $\tau\in [0,T)$ and $r>0$, we set up three admissible control sets:
\begin{description}
 \item ~~~$\mathcal U_{M,\tau}\triangleq\{u\in L^{\infty}((0,T);L^2(\Omega));~\Vert u(t)\Vert\leq M,~ a.e.~ t\in (\tau,T)\}$;
 \item ~~~$\mathcal V_{M,r}\triangleq\{u;~\exists \tau\in[0,T)~s.t.~u\in \mathcal U_{M,\tau}\mbox{ and }y(T;\chi_{(\tau,T)}u,y_0)\in B(z_d,r) \}$;
 \item ~~~$\mathcal W_{\tau,r}\triangleq\{u\in L^{\infty}((0,T);L^2(\Omega));~ y(T;\chi_{(\tau,T)}u,y_0)\in B(z_d,r) \}$.
\end{description}
Notice that the above sets may be empty. {\it We will focus ourself on the cases where they are not empty}. These cases  correspond  to some parameters $M$, $\tau$ and $r$,  and will be discussed in Section 3. When $u\in\mathcal V_{M,r}$,  we let
\begin{equation*}
 \tau_{M,r}(u)\triangleq\sup\{\tau\in[0,T);~y(T;\chi_{(\tau,T)}u,y_0)\in B(z_d,r)\}.
\end{equation*}
Now, we define optimal target control problems, second type of optimal time control problems and optimal norm control problems as follows:
\begin{description}
 \item ~~~$(OP)^{M,\tau}$: $r(M,\tau)\triangleq\inf\{\|y(T;\chi_{(\tau,T)}u,y_0)-z_d\|;~u\in \mathcal U_{M,\tau} \}$, when $\mathcal U_{M,\tau}\neq\emptyset$;
 \item ~~~$(TP)^{M,r}$: $\tau(M,r)\triangleq\sup\{\tau_{M,r}(u);~u \in \mathcal V_{M,r}\}$, when $\mathcal V_{M,r}\neq\emptyset$;
 \item ~~~$(NP)^{\tau,r}$: $M(\tau,r)\triangleq\inf\{\|\chi_{(\tau,T)}u\|_{L^{\infty}((0,T);L^2(\Omega))};~u\in \mathcal W_{\tau,r}\}$, when $\mathcal W_{\tau,r}\neq\emptyset$.
\end{description}

\noindent The values  $r(M,\tau)$, $\tau(M,r)$ and $M(\tau,r)$ are called the optimal distance, the optimal time and the optimal norm  to $(OP)^{M,\tau}$,  $(TP)^{M,r}$ and $(NP)^{\tau,r}$ respectively. Furthermore, a control $u^\ast$$\in \mathcal U_{M,\tau}$, satisfying $\|y(T;\chi_{(\tau,T)}u^\ast,y_0)-z_d\|=r(M,\tau)$, is called an optimal control to $(OP)^{M,\tau}$; a control $u^\ast$$\in \mathcal V_{M,r}$, satisfying $y(T;\chi_{(\tau(M,r),T)}u^*,y_0)\in B(z_d,r)$, is called an optimal control to $(TP)^{M,r}$; and a control $u^\ast$$\in \mathcal W_{\tau,r}$, satisfying $\|\chi_{(\tau,T)}u^*\|_{L^{\infty}((0,T);L^2(\Omega))} = M(\tau,r)$, is called an optimal control to $(NP)^{\tau,r}$.

In the studies of the equivalence of these three optimal control problems, we need that the  optimal distance  $r(M,\tau)$ is positive. To ensure it, we need impose that $M\in [0, M^\tau)$.
Here  $M^\tau$ is a function of $\tau$ defined by
\begin{equation}\label{bb}
 M^\tau=\inf\{\|\chi_{(\tau,T)}u\|_{L^{\infty}((0,T);L^2(\Omega))} ;~ u\in \mathcal{S}^\tau\} \mbox{ for all } \tau\in[0, T),
\end{equation}
where
$$
\mathcal{S}^\tau=\{u\in L^{\infty}((0,T);L^2(\Omega)),~y(T;\chi_{(\tau,T)}u,y_0)=z_d\}.
$$
By the $L^\infty$-exact controllability of Schr\"{o}dinger equations, each $\mathcal{S}^\tau$, with $\tau\in [0, T)$, is non-empty.

The first main result of this paper is stated as follows.

\begin{Theorem}\label{Theorem 1.2}
 Let $y_0\neq 0$, $M > 0$ and $T > 0$. The optimal control to $(TOCP)_M$ when  restricted in $(0,T_M)$ is an optimal control to $(NOCP)_{T_M}$; while the optimal control to $(NOCP)_T$ when extended by zero outside $(0,T)$ is an optimal control to $(TOCP)_{M_T}$.
\end{Theorem}
 Theorem~\ref{Theorem 1.2} builds up an equivalence between the minimal time and the minimal norm control problems. Such an equivalence for heat equations has been established  in  \cite{2}. From \cite{2}, it is sufficient to hold the equivalence when one has $(i)$ the $L^\infty$-null controllability in any small interval and with the control bounded by $C(T,\Omega,\omega)\|y_0\|$ (where $C(T,\Omega,\omega)$ is a positive constant depending only on $T$, $\Omega$ and $\omega$); $(ii)$ the existence of optimal controls to minimal time and norm control problems; and $(iii)$ the bang-bang property of the minimal time control problems. Fortunately, these conditions hold  for Schr\"{o}dinger equations. In fact, the condition $(i)$ follows from Theorem 4.4 in \cite{20} and Proposition 2.1 in \cite{4} ; the condition $(ii)$ follows from $(i)$, Theorem 3.3 in \cite{6} and Lemma \ref{Lemma zhang} in section 2;   and $(iii)$ follows from Proposition 4.4 in \cite{4}.
 Thus we can prove Theorem~\ref{Theorem 1.2} by a very similar way as that in the proof of Theorem 1.1 in  \cite{2}. \vskip 10pt

The second main result of this paper is as follows:
\begin{Theorem}\label{Theorem 1.3}
If  $r\in(0,r_T)$ where $r_T$ is given by (\ref{wang1.4}), and $(\tau, M)\in [0,T)\times (0,M^\tau)$  where $M^\tau$ is given by (\ref{bb}), then
 \begin{description}
  \item $(i)$  the problems $(OP)^{M,\tau}$, $(TP)^{M,r(M,\tau)}$ and $(NP)^{\tau,r(M,\tau)}$ have the same optimal control;
  \item $(ii)$  the problems $(NP)^{\tau,r}$, $(OP)^{M(\tau,r),\tau}$ and $(TP)^{M(\tau,r),r}$ have the same optimal control.
 \end{description}
If  $M >0$ and $r\in [r(M,0),r_T)\cap(0,r_T)$, then it  holds that
 \begin{description}
  \item $(iii)$  the problems $(TP)^{M,r}$, $(NP)^{\tau{(M,r)},r}$ and $(OP)^{M,\tau{(M,r)}}$ have the same optimal control.
 \end{description}
\end{Theorem}
Theorem~\ref{Theorem 1.3} establishes an equivalence among three optimal control problems $(OP)^{M,\tau}$,  $(TP)^{M,r}$ and $(NP)^{\tau,r}$. A similar equivalence for heat equations has been built up in  \cite{1}. From \cite{1}, we see that the condition that the optimal distance is positive plays a key role to derive the equivalence. To guarantee this condition, the authors in  \cite{1}
      imposed the assumption that the target $z_d$ is not in the attainable set of the controlled system. This assumption is natural for heat equations since the heat flow has the smooth effective property.  However, one cannot impose this assumption on Schr\"{o}dinger equations since Schr\"{o}dinger equations  have the $L^\infty$-exact controllability. We pass this barrier by some properties of  the function $M^\tau$, $\tau\in [0,T)$ (see subsection 3.2).

About works on the exact controllability for Schr\"{o}dinger equations, we would like quote the papers \cite{20,15,14,19,8,16} and the references therein. For time optimal control problems, it deserves to mention the papers  \cite{9,4,11,18,12,13,7,6,10,17,5} and the reference therein.

The rest of the paper is organized as follows: Section 2 presents  the proof of Theorem \ref{Theorem 1.2} and some properties of the function $M^\tau$. Section 3 provides some properties of  three optimal control problems $(OP)^{M,\tau}$,  $(TP)^{M,r}$ and $(NP)^{\tau,r}$ and proves Theorem \ref{Theorem 1.3}.


\bigskip
\section{Equivalence of $(TOCP)_M$ and $(NOCP)_T$}

  $~~~\;$ First of all, it follows from Theorem 3.3 in \cite{6} that
  the problem $(TOCP)_{M}$, with $M>0$, has optimal controls. And according to Proposition 4.4 in \cite{4}, we know $(TOCP)_{M}$ holds bang-bang property and has the unique optimal control. The existence of optimal controls to $(NOCP)_T$ is stated as follows:
\begin{Lemma}\label{Lemma zhang}
 Let $T>0$. Then the problem $(NOCP)_T$ has optimal controls.
\end{Lemma}
\begin{proof}
According to the definition of $M_T$, there is a sequence $\{u_n\}\subset\mathcal V_T$ such that
\begin{equation*}
 y(T;u_n)=0 \mbox{ and } \|u_n\|_{L^\infty(\mathbb R^+;L^2(\Omega))} \rightarrow M_T.
\end{equation*}
 Since $\{u_n\}$ is bounded, we can find a subsequence $\{u_{n_k}\}\subset\{u_n\}$ and $\tilde u\in L^\infty(\mathbb R^+;L^2(\Omega))$ such that
 \begin{equation*}
  u_{n_k} \rightarrow \tilde u \mbox{ weakly star in } L^\infty(\mathbb R^+;L^2(\Omega)).
 \end{equation*}
 This implies that
 \begin{equation*}
  \|\tilde u\|_{L^\infty(\mathbb R^+;L^2(\Omega))} \leq\displaystyle\liminf_{k\rightarrow\infty} \|u_{n_k}\|_{L^\infty(\mathbb R^+;L^2(\Omega))}=M_T
 \end{equation*}
 and
 \begin{equation*}
  y(T;u_{n_k}) \rightarrow y(T;\tilde u) \mbox{ weakly star in } L^2(\Omega).
 \end{equation*}
 Hence $y(T;\tilde u)=0$. From the above, we know $\tilde u$ is an optimal control to $(NOCP)_T$.
\end{proof}

For simplicity, we write the maps $M\rightarrow T_M$ and $T\rightarrow M_T$  as  $T_M$ and $M_T$ respectively.
\begin{Lemma}\label {Lemma 1.1}
 Let $y_0\neq 0$. Then the  map $T_M:\mathbb R^+ \longrightarrow \mathbb R^+$ is strictly monotonically decreasing and continuous. Besides, it holds that $\displaystyle\lim_{M\rightarrow 0^+}T_M=\infty~ and~ \displaystyle\lim_{M\rightarrow \infty}T_M=0$. Furthermore, it stands that
 \begin{align}
  T &=T_{M_T}\hspace{1em} ~for~ each~ T\in\mathbb R^+, \label{2.4}\\
  M &=M_{T_M}\hspace{1em} for ~each~ M\in\mathbb R^+. \label{2.5}
 \end{align}
 Consequently, the inverse of the map   $T_M$ is the map  $M_T$.
\end{Lemma}

\begin{proof} We organize the proof by several steps.

{\it Step 1. The map $T_M$ is strictly monotonically decreasing.}

Let $0<M_1 <M_2 $. We claim that $T_{M_1} > T_{M_2}$. Otherwise, we could have that $T_{M_1} \leq T_{M_2}$.
Thus, when $u_1$ is the optimal control  to $(TOCP)_{M_1}$,  the control $\tilde u\triangleq \chi_{(0,T_{M_1})}u_1$ is an optimal control to $(TOCP)_{M_2}$. Hence, it holds that
$$y(T_{M_2};\tilde u)=y(T_{M_1};\tilde u)=0 \mbox{ and } \|\tilde u\|_{L^\infty(\mathbb R^+;L^2(\Omega))} \leq M_1 < M_2. $$
By the bang-bang property of $(TOCP)_{M_2}$ (see Proposition 4.4 in \cite{4}), it follows that $\|\tilde u\|_{L^\infty(\mathbb R^+;L^2(\Omega))} = M_2$, which leads to a  contradiction.

{\it Step 2. The map $T_M$ is right-continuous.}

Let $M_1 > M_2 > \cdots >M_n > \cdots > M > 0$ and $\displaystyle\lim_{n\rightarrow\infty}M_n=M$. It suffices to show $\displaystyle\lim_{n\rightarrow\infty}T_{M_n}=T_M$.
If it did not hold, then  by the strictly decreasing monotonicity of the map $T_M$, we would have that $\displaystyle\lim_{n\rightarrow\infty}T_{M_n}=T_M-\delta$ for some $\delta > 0$. Clearly, the optimal control $u_n$ to $(TOCP)_{M_n}$ satisfies that  $y(T_{M_n};u_n)=0$ and $\|u_n\|_{L^{\infty}(\mathbb R^+;L^2(\Omega))} \leq M_n\leq M_1$ for $n\in\mathbb N^+$. Thus,  there exist a subsequence $\{u_{n_k}\}\subset\{u_n\}$ and $\tilde u\in L^{\infty}(\mathbb R^+;L^2(\Omega))$ such that
$$\chi_{(0,T_{M_{n_k}})}u_{n_k} \rightarrow \tilde u  \mbox{~weakly~star~in}~L^\infty(\mathbb R^+;L^2(\Omega)).$$
Then we have
$$y(T_M-\delta;\chi_{(0,T_{M_{n_k}})}u_{n_k})\rightarrow y(T_M-\delta;\tilde u) \mbox{ weakly in }L^2(\Omega)$$
and
$$\|\tilde u\|_{L^\infty(\mathbb R^+;L^2(\Omega))} \leq \displaystyle\liminf_{k\rightarrow\infty} \|u_{n_k}\|_{L^\infty(\mathbb R^+;L^2(\Omega))}=M.$$
So $\|y(T_M-\delta;\tilde u)\| \leq \displaystyle\liminf_{k\rightarrow\infty} \|y(T_M-\delta;\chi_{(0,T_{M_{n_k}})}u_{n_k})\|=0$. This  contradicts with the optimality of $T_M$.

{\it Step 3. The map $T_M$ is left-continuous.}

Let $0< M_1 < M_2 < \cdots <M_n < \cdots < M $ and $ \displaystyle\lim_{n\rightarrow\infty}M_n=M$. We claim that $\displaystyle\lim_{n\rightarrow\infty}T_{M_n}=T_M$. Otherwise, by the strictly monotonicity of the map $T_M$, we would have that $\displaystyle\lim_{n\rightarrow\infty}T_{M_n}=T_M+\delta$ for some $\delta > 0$. Then it holds that
\begin{equation}\label{2.1}
 T_{M_n} > T_M+\delta,~n=1,2,\cdots .
\end{equation}
Denote by $u^*(\cdot)$ and $y^*(\cdot)$ the optimal control and optimal state to $(TOCP)_M$ respectively.
For each $n\in \mathbb{N}^+$, we define
$$w_n(t)=\frac{M_n}{M}\chi_{(0,T_M)}(t)u^*(t),\;\; t\in [0,T]\;\;\mbox{ and }\;\;z_n(t)=\frac{M_n}{M}y^*(t),\;\;t\in \mathbb{R}^+.$$
It is clear that
\begin{eqnarray*}
 \left\{
  \begin{array}{lll}
   \partial_t z_n+i\Delta z_n=\chi_{\omega}w_n  &\mbox{in}&\Omega\times\mathbb R^+,\\
   z_n=0  &\mbox{on} &\partial \Omega\times \mathbb R^+,\\
   z_n(0)=\frac{M_n}{M}y_0,z_n(T_M)=0  &\mbox{in} &\Omega
  \end{array}
 \right.
\end{eqnarray*}
and
$$\|w_n\|_{L^\infty(\mathbb R^+;L^2(\Omega))} \leq M_n \mbox{ for } n\in\mathbb N^+.$$
Consider the following equation:
\begin{eqnarray*}
 \left\{
  \begin{array}{lll}
   \partial_t p_n+i\Delta p_n=\chi_{\omega}\chi_{(T_M,T_M+\delta)}w^\prime_n  &\mbox{in}&\Omega\times\mathbb R^+,\\
   p_n=0  &\mbox{on} &\partial \Omega\times \mathbb R^+,\\
   p_n(0)=(1-\frac{M_n}{M})y_0  &\mbox{in} &\Omega.
  \end{array}
 \right.
\end{eqnarray*}
By the  $L^\infty-$exact controllability for Schr\"{o}dinger equations, there exists a positive constant $C(T_M,\delta)$ such that for each $n$, there is  a control $w^{\prime\prime}_n$ satisfying
$$\|w^{\prime\prime}_n\|_{L^\infty(\mathbb R^+;L^2(\Omega))} \leq C(T_M,\delta)(1-\frac{M_n}{M}) \|y_0\|$$
and
 $$p_n(T_M+\delta)=0.$$
Then we can find an $n_0\in\mathbb N^+$ such that
$$C(T_M,\delta)(1-\frac{M_n}{M}) \|y_0\| \leq M_n,\;\;\mbox{when}\;\; n \geq n_0.$$
Set
$$\tilde u_{n_0}=\chi_{(0,T_M)}w_{n_0}+\chi_{(T_M,T_M+\delta)}w^{\prime\prime}_{n_0}\;\;\mbox{over}\;\; \mathbb{R}^+.$$
Clearly,
$$\|\tilde u_{n_0}\|_{L^\infty(\mathbb R^+;L^2(\Omega))} \leq M_{n_0}$$
and
$$y(\cdot;\tilde u_{n_0})=z_{n_0}(\cdot;w_{n_0})+p_{n_0}(\cdot;w^{\prime\prime}_{n_0})~ \mbox{and}~ y(T_M+\delta;\tilde u_{n_0})=0.$$
Hence,  $\tilde u_{n_0}\in \mathcal U_{M_{n_0}}$ and $T_{M_{n_0}} \leq T_M+\delta$. These contradict to (\ref{2.1}).

{\it Step 4. It holds that $\displaystyle \lim_{M\rightarrow 0^+} T_M=\infty$.}

If this didn't stand, then, according to the monotonicity of the map $T_M$, we would have that when $M >0$,  $T_M < T_0$ for some positive number $T_0 > 0$. Let $\{M_n\}\subset\mathbb R^+$ be a sequence  satisfying $\displaystyle\lim_{n\rightarrow\infty}M_n =0$. Write $u_n$ for an optimal control to $(TOCP)_{M_n}$. Then we have
\begin{equation*}
 \|u_n\|_{L^\infty(\mathbb R^+;L^2(\Omega))} \leq M_n < M_0 \mbox{ for some }M_0 >0 \mbox{ and all } n\in\mathbb N^+.~
\end{equation*}
From these,  there exist a subsequence $\{u_{n_k}\}\subset\{u_n\}$ and $\tilde u\in L^{\infty}(\mathbb R^+;L^2(\Omega))$ such that
\begin{equation}\label{aa}
 \chi_{(0,T_{M_{n_k}})}u_{n_k}\rightarrow \tilde u \mbox{~weakly~star~in}~L^\infty(\mathbb R^+;L^2(\Omega)).
\end{equation}
So it follows that
\begin{equation*}
 \|\tilde u\|_{L^\infty(\mathbb R^+;L^2(\Omega))} \leq \displaystyle\liminf_{k\rightarrow\infty}\|u_{n_k}\|_{L^\infty(\mathbb R^+;L^2(\Omega))} \leq \displaystyle\liminf_{k\rightarrow\infty} M_{n_k}=0.
\end{equation*}
This implies that  $\tilde u=0$. On the other hand, by (\ref{aa}) and the fact that $T_M < T_0$, we get
\begin{equation*}
 y(T_0;\chi_{(0,T_{M_{n_k}})}u_{n_k}) \rightarrow y(T_0;\tilde u) \mbox{ weakly in } L^2(\Omega).
\end{equation*}
So $\|y(T_0;\tilde u)\| \leq \displaystyle\liminf_{k\rightarrow\infty}\|y(T_0;\chi_{(0,T_{M_{n_k}})}u_{n_k})\|$. This, along with the fact that $y(T_{M_{n_k}};u_{n_k})=0$, indicates that
 $y(T_0;\tilde u)=0$, namely,  $y(T_0;0)=0$, which  contradicts to the assumption that  $y_0\neq 0$.

{\it Step 5. $\displaystyle\lim_{M\rightarrow\infty} T_M = 0$.}

If it didn't hold, then by the monotonicity of the map $T_M$, we could have that $\displaystyle\lim_{M\rightarrow\infty}T_M > 2T_0$ for some $T_0>0$. This, together with the monotonicity of $T_M$, yields that
 $T_M > T_0$ when $M > 0$. According to the $L^\infty$-exact controllability for Sch\"{o}dinger equations, there exist a constant $C(T_0)$ and a control $\hat u$ with $\|\hat u\|_{L^\infty(\mathbb R^+;L^2(\Omega))} \leq C(T_0)\|y_0\|$ such that $y(T_0;\hat u)=0$. Thus, by the definition of the minimal time $T_M$, we have that $T_M \leq T_0$ when $M > C(T_0)\|y_0\|$. This contradicts to the assumption that $\displaystyle\lim_{M\rightarrow\infty}T_M > 2T_0$.

{\it Step 6. $T_{M_T}=T \mbox{ for } T\in\mathbb R^+$.}

Let $T > 0$ and $v_1$ be an optimal control to $(NOCP)_T$ (see Lemma \ref{Lemma zhang}). Then we have
\begin{equation*}
 y(T;v_1)=0 \mbox{~and~} \|v_1\|_{L^{\infty}((0,T);L^2(\Omega))}=M_T.
\end{equation*}
We extend  $v_1$ over $\mathbb R^+$ by setting it to be zero over $(T,\infty)$, and denote this extension  by $\tilde v_1$. It is clear that  $\tilde v_1 \in \mathcal U_{M_T}$. Then according to the definition of the optimal time $T_{M_T}$ to $(TOCP)_{M_T}$, it holds that $T_{M_T} \leq T$. We claim that $T_{M_T} = T$. Seeking for a contradiction, we could assume that $T_{M_T} < T$. Then by the strictly deceasing monotonicity and continuity of the map $T_M$, there would be a positive number $M_1$ with $0 < M_1 < M_T$ such that $T_{M_1} =T$. We denote by $u_2$ the optimal control to $(TOCP)_{M_1}$. Then it holds that
\begin{equation*}
 y(T;u_2)=0 \mbox{~and~} \|u_2\|_{L^\infty(\mathbb R^+;L^2(\Omega))}\leq M_1.
\end{equation*}
Clearly $u_2|_{(0,T)} \in \mathcal V_T$ and $M_1 < M_T$. These  contradict with the optimality of  $M_T$ to $(NOCP)_T$.

{\it Step 7. $M_{T_M}=M$ when $M\in\mathbb R^+$.}

Let $0 < M < \infty$. Using the fact $T_{M_T}=T$ for every $T > 0$ in step 6 and setting $T=T_M$, we have $T_{M_{T_M}}=T_M$ since $T_M > 0$. Then by the strictly monotonicity of the map $T_M$, it must hold that $M_{T_M}=M$.
\end{proof}

Now, we prove Theorem \ref{Theorem 1.2}.
\begin{proof}[Proof of Theorem \ref{Theorem 1.2}]
 Let $M > 0$ and $u_1$ be an optimal control to $(TOCP)_M$. Then it holds that
\begin{equation}\label{ab}
 y(T_M;u_1)=0 \mbox{~and~} \|u_1\|_{L^\infty(\mathbb R^+;L^2(\Omega))}\leq M.
\end{equation}
So $u_1|_{(0,T_M)} \in \mathcal V_{T_M}$. By (\ref{ab}) and the fact that $M=M_{T_M}$ \big(see (\ref{2.5})\big), we have
\begin{equation*}
 y(T_M;u_1|_{(0,T_M)})=0 \mbox{~and~} \big\|u_1|_{(0,T_M)}\big\|_{L^\infty((0,T_M);L^2(\Omega))}\leq M=M_{T_M}.
\end{equation*}
Then the control $u_1|_{(0,T_M)}$ is an optimal control to $(NOCP)_{T_M}$. \\
Let $T > 0$ and $v_2$ be an optimal control to $(NOCP)_T$. Then one can deduce that
\begin{equation}\label{ac}
 y(T;v_2)=0 \mbox{~and~} \|v_2\|_{L^\infty((0,T);L^2(\Omega))}=M_T.
\end{equation}
Extend $v_2$ over $\mathbb R^+$ by setting it to be zero outside $(0,T)$ and denote this extension  by $\tilde v_2$. Then we have $\tilde v_2 \in \mathcal U_{M_T}$.
By (\ref{ac}) and the fact that $T_{M_T}=T$ \big(see (\ref{2.4})\big), we obtain that
\begin{equation*}
 y(T_{M_T};\tilde v_2)=y(T;\tilde v_2)=0 \mbox{~and~} \|\tilde v_2\|_{L^\infty(\mathbb R^+;L^2(\Omega))}=M_T.
\end{equation*}
Thus,  the control $\tilde v_2$ is an optimal control to $(TOCP)_{M_T}$.
\end{proof}
\vskip 10pt


\bigskip
\section{Equivalence of three optimal control problems}

$~~~$ In this section, we fix  $T>0$  and simply write  $\|\cdot\|_\infty$ for the norm of  $L^\infty((0,T);L^2(\Omega))$.

\subsection{Properties of the function $\tau\rightarrow M^\tau$, $\tau\in [0,T)$}

Recall that   the function $M^\tau$  is defined by (\ref{bb}).
Then for each $\tau\in [0,T)$, $M^\tau$ gives the following optimal norm control problem $(NP)^{\tau,0}$:
$$M^\tau=\inf\{\|\chi_{(\tau,T)}u\|_{L^{\infty}((0,T);L^2(\Omega))};~ u\in L^{\infty}((0,T);L^2(\Omega)),~ y(T;\chi_{(\tau,T)}u,y_0)=z_d\},$$
where $y(\cdot;\chi_{(\tau,T)}u,y_0)$ satisfies Equation (\ref{1.1}) with initial data $y_0$. \\
Set $z_{d^\prime}=z_d-y(T;0,y_0)$. Since $r_T=\|y(T;0,y_0)-z_d\| > 0$, it holds that  $z_{d^\prime} \neq 0$.
Clearly, the problem $(NP)^{\tau,0}$ is equivalent to the following  problem:
$$M_{1}^{\tau}\triangleq\inf\{\|\chi_{(\tau,T)}u\|_{L^{\infty}((0,T);L^2(\Omega))};~ u\in L^{\infty}((0,T);L^2(\Omega)),~ y(T;\chi_{(\tau,T)}u,0)=z_{d^\prime}\},$$
 where $y(\cdot;\chi_{(\tau,T)}u,0)$ satisfies Equation (\ref{1.1}) where $y_0=0$.
 One can easily check that $M_{1}^\tau=M^\tau$.
We define  $$\bar z(t)=y(T-t;\chi_{(0,T-\tau)}u(T-t),0)\;\;\mbox{and}\;\;\bar v(t)= -\bar{u}(T-t),\;\; t\in [0,T].$$
It is clear that
\begin{eqnarray}\label{1.3}
 \left\{
  \begin{array}{lll}
   \partial_t z+i\Delta z=\chi_{\omega}\chi_{(0,T-\tau)}v  &\mbox{in}&\Omega\times(0,T),\\
   z=0  &\mbox{on} &\partial \Omega\times (0,T),\\
   z(0)=z_{d^\prime},z(T-\tau)=0  &\mbox{in} &\Omega.
  \end{array}
 \right.
\end{eqnarray}
Then, the problem $(NP)^{\tau,0}$ equals to the problem:
$$M_{2}^{\tau}\triangleq\inf\{\|v\|_{L^{\infty}((0,T-\tau);L^2(\Omega))};~ v\in L^{\infty}((0,T);L^2(\Omega)),~ z(T;\chi_{(0,T-\tau)}v,z_{d^\prime})=0\},$$
where $z(t;\chi_{(0,T-\tau)}v,z_{d^\prime})$ verifies Equation (\ref{1.3}).
In summary, we conclude that the functions $M^\tau$ and  $M_{2}^{\tau}$ are the same.

\begin{Lemma}\label{Lemma Mtau}
 Let $\tau\in [0,T)$. The map $\tau\longrightarrow M^\tau$ is strictly monotonically increasing and continuous from $[0,T)$ to $[M^0,\infty)$. Moreover, it holds that $\displaystyle\lim_{\tau\rightarrow T^-} M^\tau=\infty$.
\end{Lemma}

\begin{proof}
First of all, we have
\begin{equation*}
 M^\tau=M_{2}^\tau,\;\;\tau\in [0,T).
\end{equation*}
 One can easily check that
\begin{equation*}
 M_2^\tau=M_{T-\tau}, \;\;\tau\in [0,T).
\end{equation*}
Thus, the map $\tau\rightarrow M^\tau$, $\tau\in [0,T)$ is the same as the map $\tau\rightarrow M_{T-\tau}$,  $\tau\in [0,T)$ . This, along with lemma \ref{Lemma 1.1}, gives the desired properties of the map $M^\tau$.
\end{proof}

\subsection{Equivalence between the optimal and norm target problems }

$~~~$ We begin with studying the properties of the optimal target problem $(OP)^{M,\tau}$.

\begin{Lemma}\label{Lemma OP}
Let $M^\tau$  be defined by (\ref{bb}). Suppose that  $(\tau, M)\in [0,T)\times [0,M^\tau)$. Then, \\
$(i)$ $(OP)^{M,\tau}$ has optimal controls;
 $(ii)$ $r(M,\tau)>0$;
$(iii)$ $u^*$ is an optimal control to $(OP)^{M,\tau}$ if and only if $u^*\in \mathcal U_{M,\tau}$ satisfies the maximal condition
   \begin{equation}\label{3.1}
    Re\int^T_0 \langle \chi_{(\tau,T)}(t)\chi_{\omega}\varphi^*(t),u^*(t) \rangle dt = \displaystyle\max_{v(\cdot)\in\mathcal U_{M,\tau}} Re\int^T_0 \langle \chi_{(\tau,T)}(t)\chi_{\omega}\varphi^*(t),v(t) \rangle dt,
   \end{equation}
   where $\varphi^*$ is the solution to the adjoint equation
   \begin{eqnarray}\label{3.2}
    \left\{
     \begin{array}{lll}
      \partial_t \varphi^*-i\Delta \varphi^*=0  &\mbox{in}&\Omega\times(0,T),\\
      \varphi^*=0  &\mbox{on} &\partial \Omega\times (0,T),\\
      \varphi^*(T)=-(y^*(T)-z_d)  &\mbox{in} &\Omega
     \end{array}
    \right.
   \end{eqnarray}
   and $y^*$ solves Equation (\ref{1.1})
   \begin{eqnarray}\label{3.3}
    \left\{
     \begin{array}{lll}
      \partial_t y^*+i\Delta y^*=\chi_{\omega}\chi_{(\tau,T)}u^*  &\mbox{in}&\Omega\times(0,T),\\
      y^*=0  &\mbox{on} &\partial \Omega\times (0,T),\\
      y^*(0)=y_0 &\mbox{in}&\Omega.
     \end{array}
    \right.
   \end{eqnarray}

\end{Lemma}

\begin{proof}
 $(i)$ According to the definition of $r(M,\tau)$, we have
 \begin{equation*}
  r(M,\tau)=\displaystyle\inf_{u\in\mathcal U_{M,\tau}} \|y(T;\chi_{(\tau,T)}u,y_0)-z_d\|.
 \end{equation*}
 Then there is a sequence $\{u_n\} \subset\mathcal U_{M,\tau}$ such that
 \begin{equation}\label{3.4}
  \|y(T;\chi_{(\tau,T)}u_n,y_0)-z_d\| \rightarrow r(M,\tau) \mbox{  \;\;as } n\rightarrow\infty.
 \end{equation}
 On the other hand, it holds that $\|\chi_{(\tau,T)}u_n\|_{\infty} \leq M$. So there exist a subsequence $\{u_{n_k}\}\subset\{u_n\}$ and $\tilde u\in L^{\infty}((0,T);L^2(\Omega))$ such that
 \begin{equation}\label{zhang ac}
  \chi_{(\tau,T)}u_{n_k} \rightarrow \tilde u \mbox{ weakly star in} ~L^{\infty}((0,T);L^2(\Omega))
 \end{equation}
 and
 \begin{equation*}
  \|\tilde u\|_{\infty} \leq \displaystyle\liminf_{k\rightarrow\infty} \|\chi_{(\tau,T)}u_{n_k}\|_{\infty} \leq M.
 \end{equation*}
 Thus, it holds that  $\tilde u \in\mathcal U_{M,\tau}$. By (\ref{zhang ac}), we can deduce that
 \begin{equation*}
  y(T;\chi_{(\tau,T)}u_{n_k},y_0)) \rightarrow y(T;\chi_{(\tau,T)}\tilde u,y_0)) \mbox{ weakly in } L^2(\Omega).
 \end{equation*}
 This, together with  (\ref{3.4}) and the definition of $r(M,\tau)$, indicates  that
 \begin{equation*}
  r(M,\tau)\leq \|y(T;\chi_{(\tau,T)}\tilde u,y_0))-z_d\| \leq \displaystyle\liminf_{k\rightarrow\infty} \|y(T;\chi_{(\tau,T)}u_{n_k},y_0))-z_d\| =r(M,\tau).
 \end{equation*}
From this and the fact that  $\tilde u \in\mathcal U_{M,\tau}$, it follows that  $\tilde u$ is an optimal control to $(OP)^{M,\tau}$.

 $(ii)$ According to  $(i)$, there is a control $u\in L^{\infty}((0,T);L^2(\Omega))$ such that
 \begin{equation*}
  \|y(T;\chi_{(\tau,T)}u,y_0)-z_d\| =r(M,\tau)  \mbox{~and~} \|\chi_{(\tau,T)}u\|_{\infty} \leq M.
 \end{equation*}
 If $r(M,\tau)=0$, then by the definition of $M^\tau$, we can obtain that $\|\chi_{(\tau,T)}u\|_{\infty} \geq M^\tau$. So it contradicts to the assumption that $M\in [0,M^\tau)$.

 $(iii)$ For every $\tilde u$, $u \in\mathcal U_{M,\tau}$ and $\varepsilon > 0$, set $u_\varepsilon=(1-\varepsilon)\tilde u + \varepsilon u$. Write $\tilde y(t)=y(t;\chi_{(\tau,T)}\tilde u,y_0)$, $y(t)=y(t;\chi_{(\tau,T)}u,y_0)$ and $y_\varepsilon(t)=y(t;\chi_{(\tau,T)}u_\varepsilon,y_0)$, respectively.
It is clear that
 \begin{eqnarray}\label{3.6}
  &\;&\|y_\varepsilon(T)-z_d\|^2 - \|\tilde y(T)-z_d\|^2 \nonumber \\
  &=&\varepsilon^2 \left\|\int^T_0 e^{-i\Delta(T-t)}\chi_{(\tau,T)}\chi_\omega(u-\tilde u) dt\right\|^2 \nonumber\\
  &\;&+ 2\varepsilon \left(Re\int^T_0 \langle e^{-i\Delta(T-t)}\chi_{(\tau,T)}\chi_\omega(u-\tilde u),\tilde y(T)-z_d \rangle dt\right) \nonumber \\
  &=&\varepsilon^2 \left\|\int^T_0 e^{-i\Delta(T-t)}\chi_{(\tau,T)}\chi_\omega(u-\tilde u) dt\right\|^2\nonumber\\
  &\;&+ 2\varepsilon \left(Re\int^T_0 \langle (\tilde u-u),\chi_{(\tau,T)}\chi_\omega e^{i\Delta(T-t)}(-(\tilde y(T)-z_d)) \rangle dt\right).
 \end{eqnarray}
Let $u^*$ be an optimal control to $(OP)^{M,\tau}$. Taking  $\tilde u=u^*$ in (\ref{3.6}), we obtain that
 \begin{equation*}
  Re\int^T_0 \langle (u^*-u),\chi_{(\tau,T)}\chi_\omega e^{i\Delta(T-t)}(-(y^*(T)-z_d)) \rangle dt \geq 0 \mbox{~for all~} u\in\mathcal U_{M,\tau}.
 \end{equation*}
This implies that
 \begin{equation*}
  Re\int^T_0 \langle \chi_{(\tau,T)}(t)\chi_{\omega}\varphi^*(t),u^*(t) \rangle dt = \displaystyle\max_{v(\cdot)\in\mathcal U_{M,\tau}} Re\int^T_0 \langle \chi_{(\tau,T)}(t)\chi_{\omega}\varphi^*(t),v(t) \rangle dt,
 \end{equation*}
 where $\varphi^*$ and $y^*$ satisfy the equations (\ref{3.2}) and (\ref{3.3}) respectively. \\
 Conversely, for every $u\in\mathcal U_{M,\tau}$ and $\varepsilon > 0$, we suppose that $\tilde u\in \mathcal U_{M,\tau}$ satisfies (\ref{3.1}). Set $u_\varepsilon=(1-\varepsilon)\tilde u + \varepsilon u$. Then by (\ref{3.6}) and (\ref{3.1}), one can easily obtain that
 \begin{eqnarray*}
  &\;&\|y_\varepsilon(T)-z_d\|^2 - \|\tilde y(T)-z_d\|^2  \\
  &=&\varepsilon^2\left\|\int^T_0 e^{-i\Delta(T-t)}\chi_{(\tau,T)}\chi_\omega(u-\tilde u) dt\right\|^2  \\
  &\;& + 2\varepsilon\left(Re\int^T_0 \langle (\tilde u-u),\chi_{(\tau,T)}\chi_\omega e^{i\Delta(T-t)}(-(\tilde y(T)-z_d)) \rangle dt\right) \\
  &\geq& 0.
 \end{eqnarray*}
Taking  $\varepsilon =1$ in the above,  we get
 \begin{equation*}
  \|y(T)-z_d\|^2 - \|\tilde y(T)-z_d\|^2 \geq 0  \mbox{~for every~}u\in\mathcal U_{M,\tau}.
 \end{equation*}
 This shows that  $\tilde u$ is an optimal control to $(OP)^{M,\tau}$.
\end{proof}

\begin{Remark}\label{Rk maximal}
 $(a)$ The conclusion $(i)$ in Lemma~\ref{Lemma OP}  holds for all $M\geq 0$, in particular, for the case where $M\geq M^\tau$; \\
 $(b)$ When we get rid of the assumption that $M< M^\tau$, a control $u\in\mathcal U_{M,\tau}$ satisfying the maximal condition (\ref{3.1}) must be an optimal control to $(OP)^{M,\tau}$. This can be deduced from (\ref{3.6}). In other words, the result $(iii)$ in Lemma \ref{Lemma OP} still stands without  the assumption that $M< M^\tau$. This can be observed from  the proof of $(iii)$.
  However, to have the conclusion $(ii)$ in Lemma \ref{Lemma OP}, we necessarily need  the assumption that $0\leq M<M^\tau$.
\end{Remark}

\begin{Lemma}\label{Lemma 3.3}
 Let $(\tau, M)\in [0,T)\times [0,M^\tau)$. Then it holds that
 \begin{description}
  \item $(i)$ $u^*$ is an optimal control to $(OP)^{M,\tau}$ if and only if $u^*\in L^{\infty}((0,T);L^2(\Omega))$ satisfies
   \begin{equation}\label{3.7}
    u^*(t) = M\frac{\chi_\omega \varphi^*(t)}{\|\chi_\omega \varphi^*(t)\|}  \mbox{ for~a.e.~} t\in (\tau,T),
   \end{equation}
   \noindent where $\varphi^*$ and $y^*$ satisfies Equation (\ref{3.2}) and Equation (\ref{3.3}) respectively;
  \item $(ii)$ $(OP)^{M,\tau}$ has the $bang-bang$ property, i.e., every optimal control $u^*$ holds that $\|u^*(t)\| = M~\mbox{ for a.e.}~t\in (\tau,T)$;
  \item $(iii)$ $(OP)^{M,\tau}$ has the unique optimal control, in the sense of that every two optimal controls $u_1$ and $u_2$ to $(OP)^{M,\tau}$ hold that $u_1(t)=u_2(t)$ for a.e. $t\in (\tau,T)$.
 \end{description}
\end{Lemma}

\begin{proof}
 $(i)$ One can easily check that  (\ref{3.1}) is equivalent to
  \begin{equation}\label{3.8}
   Re\langle \chi_{\omega}\varphi^*(t),u^*(t) \rangle = \displaystyle\max_{v\in L^2(\Omega),\|v\| \leq M} Re\langle \chi_{\omega}\varphi^*(t),v \rangle  \mbox{ for a.e. }  t\in (\tau,T).
  \end{equation}
 For each $v\in L^2(\Omega)$ with $\|v\| \leq M$, we set  $v_1= v\cdot sgn(\langle \chi_{\omega}\varphi^*(t), v \rangle)$. Clearly, $\|v_1\| \leq M$. We replace $v$ by $v_1$ in
 (\ref{3.8}) to get
   \begin{equation*}
   Re\langle \chi_{\omega}\varphi^*(t),u^*(t) \rangle = \displaystyle\max_{v\in L^2(\Omega),\|v\| \leq M} |\langle \chi_{\omega}\varphi^*(t),v \rangle|  \mbox{ for a.e. }  t\in (\tau,T).
  \end{equation*}
By the above identity and  the fact that $\|u^*(t)\| \leq M$ for a.e. $t\in (\tau,T)$, we can find that $Im \langle \chi_{\omega}\varphi^*(t),u^*(t) \rangle =0$ and
  \begin{equation}\label{3.9}
   \langle \chi_{\omega}\varphi^*(t),u^*(t) \rangle = \displaystyle\max_{v\in L^2(\Omega),\|v\| \leq M} |\langle \chi_{\omega}\varphi^*(t),v \rangle|  \mbox{ for a.e. }  t\in (\tau,T).
  \end{equation}
  On the other hand, because $\varphi^*(T)= -(y^*(T)-z_d)\neq 0$ (see $(ii)$ in Lemma \ref{Lemma OP}) and the unique continuation of Sch\"{o}dinger equations (see Lemma 4.1 and Lemma 4.3 in \cite{4}), we get
  \begin{equation*}
   \chi_\omega \varphi^*(t)\neq 0  \mbox{ for a.e. }  t\in (\tau,T).
  \end{equation*}
  From this and (\ref{3.9}), one can easily deduce that
  \begin{equation*}
    u^*(t) = M\frac{\chi_\omega \varphi^*(t)}{\|\chi_\omega \varphi^*(t)\|}  \mbox{ for a.e. }  t\in (\tau,T).
  \end{equation*}

 $(ii)$ follows from  $(i)$ at once.

 $(iii)$ Let $u_1$ and $u_2$ be two optimal controls to $(OP)^{M,\tau}$. Clearly,  the control $\frac{u_1+u_2}{2}$ is also an optimal control to this problem.  From the bang-bang property of $(OP)^{M,\tau}$, we have
 \begin{equation}\label{3.10}
  \|u_1(t)\| = \|u_2(t)\| = \left\|\frac{u_1(t)+u_2(t)}{2}\right\| = M \mbox{ for a.e. }  t\in (\tau,T).
 \end{equation}
 According to the parallelogram law and (\ref{3.10}),  it holds that
 \begin{eqnarray*}
  \left\|\frac{u_1(t)-u_2(t)}{2}\right\|^2 &=& \frac{1}{2} (\|u_1(t)\|^2 + \|u_2(t)\|^2) - \left\|\frac{u_1(t)+u_2(t)}{2}\right\|^2 \\
                              &\leq & \frac{1}{2} (M^2 + M^2) - M^2 =0 \;\;\mbox{ for a.e. }\; t\in (\tau,T).
 \end{eqnarray*}
 So we have $u_1(t)=u_2(t)$ for a.e. $t\in (\tau,T)$. This completes the proof.
\end{proof}

Next, we study the existence of optimal control to the  problem $(NP)^{\tau,r}$.
\begin{Lemma}\label{Np}
 Let $\tau\in [0,T)$ and $r\in [0,\infty)$. Then the optimal norm problem $(NP)^{\tau,r}$ has optimal controls and  it holds that $M(\tau,r) \leq M^\tau$. Moreover, $M(\tau,r)>0$ when $r<r_T$.
\end{Lemma}

\begin{proof}
We first notice  the non-emptiness of   the admissible control set $\mathcal W_{\tau,r}$, since  the $L^\infty$-null controllability for Sch\"{o}dinger equations holds under the assumption that $\omega$ satisfies the geometric control condition.

Next, according to the definition of  $M(\tau,r)$ to $(NP)^{\tau,r}$, we have
 \begin{equation*}
  M(\tau,r) = \displaystyle\inf_{u\in\mathcal W_{\tau,r}} \{\|\chi_{(\tau,T)}u\|_{\infty}; y(T;\chi_{(\tau,T)}u,y_0)\in B(z_d,r)\}.
 \end{equation*}
 From this, there is a sequence $\{u_n\}\subset \mathcal W_{\tau,r}$ such that
 \begin{equation}\label{zhang ab}
  \|y(T;\chi_{(\tau,T)}u_n,y_0)-z_d\| \leq r \mbox{ and } \displaystyle\lim_{n\rightarrow\infty} \|\chi_{(\tau,T)}u_n\|_{\infty} = M(\tau,r).
 \end{equation}
 Then  there exist a subsequence $\{u_{n_k}\}\subset\{u_n\}$ and $\tilde u\in L^{\infty}((0,T);L^2(\Omega))$ such that
 \begin{equation}\label{3.11}
  \chi_{(\tau,T)}u_{n_k} \longrightarrow \tilde u \mbox{ weakly star in} ~L^{\infty}((0,T);L^2(\Omega))
 \end{equation}
 and
 \begin{equation}\label{3.12}
  \|\tilde u\|_{\infty} \leq \displaystyle\liminf_{k\rightarrow\infty} \|\chi_{(\tau,T)}u_{n_k}\|_{\infty} \leq M(\tau,r).
 \end{equation}
 By (\ref{3.11}), we have
 \begin{equation*}
  y(T;\chi_{(\tau,T)}u_{n_k},y_0) \rightarrow y(T;\chi_{(\tau,T)}\tilde u,y_0) \mbox{ weakly in } L^2(\Omega).
 \end{equation*}
 From this and (\ref{zhang ab}), we have
 \begin{equation*}
   \|y(T;\chi_{(\tau,T)}\tilde u,y_0) - z_d\| \leq \displaystyle\liminf_{k\rightarrow\infty} \|y(T;\chi_{(\tau,T)}u_{n_k},y_0) - z_d\| \leq r.
 \end{equation*}
 This, together with (\ref{3.12}), indicates that $\tilde u$ is an optimal control to $(NP)^{\tau,r}$.

 Furthermore, by the definition of $M^\tau$ and the existence of optimal controls to $(NP)^{\tau,0}$, there is a control $\hat u\in L^{\infty}((0,T);L^2(\Omega))$ such that
 \begin{equation}\label{zhang3.13}
  y(T;\chi_{(\tau,T)}\hat u,y_0) = z_d \mbox{ and } \|\chi_{(\tau,T)}\hat u\|_\infty = M^\tau.
 \end{equation}
 From the first equality above, we know that  $\hat u\in \mathcal W_{\tau,r}$. This, along with (\ref{zhang3.13}) and  the optimality of $M(\tau,r)$, implies that $M(\tau,r) \leq M^\tau$.

 Finally, we are going to show that $M(\tau,r)>0$ when $r<r_T$.
 Seeking for a contradiction, we could suppose that $M(\tau,r)=0$ for some $r<r_T$. Then the optimal control   $u_\tau$ to $(NP)^{\tau,r}$ would satisfy that $u_\tau=0$ and $\|y(T;\chi_{(\tau,T)}u_\tau, y_0)-z_d\|\leq r$. Thus, it holds that  $\|y(T;\chi_{(\tau,T)}u_\tau,y_0)-z_d\|=\|y(T;0,y_0)-z_d\|=r_T>r$ (see (\ref{wang1.4})), which leads to a contradiction.
\end{proof}

\begin{Remark}
In fact, when $r > 0$, the conclusion that $M(\tau,r) \leq M^\tau$ in Lemma~\ref{Np}  can be improved to that $M(\tau,r) < M^\tau$.
This will be proved in  Lemma \ref{Lemma 3.5} and will be used later.
\end{Remark}

\begin{Lemma}\label{Lemma 3.5}
 Let $\tau\in [0,T)$. The map $r(\cdot,\tau): [0,M^\tau)\longrightarrow (0,r_T]$ is strictly monotonically decreasing and Lipschitz continuous and verifies that $r(0,\tau)=r_T$ and $\displaystyle\lim_{M\rightarrow M^\tau} r(M,\tau) = 0$. Then $M(\tau,r)< M^\tau$ when $r >0$. Besides, it holds that
 \begin{equation}\label{3.13}
  r=r(M(\tau,r),\tau) \mbox{ when } r\in (0,r_T]
 \end{equation}
 and
 \begin{equation}\label{3.14}
  M=M(\tau,r(M,\tau)) \mbox{ when } M\in [0,M^\tau).
 \end{equation}
 Consequently, the inverse of the map $r(\cdot,\tau)$ is the map $M(\tau,\cdot)$.
\end{Lemma}

\begin{proof}
We carry out the proof by several steps.

{\it  Step 1. The map $r(\cdot,\tau)$ is strictly monotonically decreasing.}

 Let $0\leq M_1 < M_2 < M^\tau$. We claim that $r(M_1,\tau) > r(M_2,\tau)$. Otherwise,  we would have that $r(M_1,\tau) \leq r(M_2,\tau)$. Then by the Lemma \ref{Lemma OP}, there is an optimal control $u^*\in\mathcal U_{M_1,\tau}$ to $(OP)^{M_1,\tau}$ such that
 \begin{equation*}
  \|y(T;\chi_{(\tau,T)}u^*,y_0) - z_d\| =r(M_1,\tau)\leq r(M_2,\tau) \mbox{ and } \|\chi_{(\tau,T)}u^*\|_\infty \leq M_1 < M_2.
 \end{equation*}
 So $u^*$ is also an optimal control to $(OP)^{M_2,\tau}$. On the other hand, by the bang-bang property of $(OP)^{M_2,\tau}$ (see $(ii)$ in Lemma \ref{Lemma 3.3}), it holds that $\|\chi_{(\tau,T)}u^*(t)\| = M_2$ for a.e. $t\in(\tau,T)$.
  This contradicts to the facts that  $u^*\in\mathcal U_{M_1,\tau}$ and $M_2>M_1$.

 {\it Step 2. The map $r(\cdot,\tau)$ is Lipschitz continuous.}

 Let $0\leq M_1 < M_2 < M^\tau$. Let  $\hat u\in\mathcal U_{M_2,\tau}$ be the optimal control to $(OP)^{M_2,\tau}$. Then  we have
 \begin{eqnarray*}
  &\;&r(M_2,\tau) = \left\|e^{-i\Delta T}y_0 + \int^T_\tau e^{-i\Delta (T-t)}\chi_\omega \hat u dt - z_d \right\|  \\
  &\geq & \left\|e^{-i\Delta T}y_0 + \int^T_\tau e^{-i\Delta (T-t)}\chi_\omega \frac{M_1}{M_2}\hat u dt - z_d \right\|  -  \frac{M_2-M_1}{M_2}\left\|\int^T_\tau e^{-i\Delta (T-t)}\chi_\omega \hat u dt \right\|.
 \end{eqnarray*}
It follows from the fact that $\frac{M_1}{M_2}\hat u\in\mathcal U_{M_1,\tau}$ and the optimality of $r(M_1,\tau)$ that
 \begin{equation*}
  \left\|e^{-i\Delta T}y_0 + \int^T_\tau e^{-i\Delta (T-t)}\chi_\omega \frac{M_1}{M_2}\hat u dt - z_d \right\| \geq r(M_1,\tau).
 \end{equation*}
It is clear that
  $$\left\|\int^T_\tau e^{-i\Delta (T-t)}\chi_\omega \hat u dt \right\| \leq M_2(T-\tau).$$
 Combining the above three estimates, we obtain that
 \begin{equation*}
  r(M_2,\tau) > r(M_1,\tau) - (M_2-M_1)(T-\tau).
 \end{equation*}
 This, together with the fact that  $r(M_2,\tau) < r(M_1,\tau)$ (see Step 1), indicates that
 \begin{equation*}
  |r(M_1,\tau) - r(M_2,\tau)| \leq (T-\tau)|M_1-M_2|.
 \end{equation*}

{\it  Step 3. $r(0,\tau)=r_T$ and $\displaystyle\lim_{M\rightarrow M^\tau} r(M,\tau)=0$.}

 First, according to the definition of $r_T$ (see (\ref{wang1.4})), it is obvious that $r(0,\tau)=r_T$.

 Next, let $0\leq M < M^\tau$. By using the definition of $M^\tau$ and Lemma \ref{Np}, we see that  there is a control $\bar u\in L^{\infty}((0,T);L^2(\Omega))$ such that
 \begin{equation*}
  y(T;\chi_{(\tau,T)}\bar u,y_0) = z_d \mbox{ and } \|\chi_{(\tau,T)}\bar u\|_\infty = M^\tau.
 \end{equation*}
 Since $\frac{M}{M^\tau}\bar u\in \mathcal U_{M,\tau}$, by the definition of $r(M,\tau)$, we get
 \begin{eqnarray*}
 &\;& r(M,\tau) \leq \|y(T;\chi_{(\tau,T)} \frac{M}{M^\tau}\bar u,y_0) - z_d\| \\
 &\leq & \left\|\left(e^{-i\Delta T}y_0 + \int^T_\tau e^{-i\Delta (T-t)}\chi_\omega \frac{M}{M^\tau}\bar u dt \right) - \left(e^{-i\Delta T}y_0 + \int^T_\tau e^{-i\Delta (T-t)}\chi_\omega \bar u dt \right) \right\|  \\
 &\leq & \frac{|M-M^\tau|}{M^\tau} (M^\tau |T-\tau|) \leq |T-\tau| |M-M^\tau|.
 \end{eqnarray*}
 So it holds that $\displaystyle\lim_{M\rightarrow M^\tau} r(M,\tau)=0$.

 {\it Step 4. It holds that $M(\tau,r)< M^\tau$ when $r >0$.}

  Let $r > 0$ and $r^0 \triangleq \min(r,r_T)$. According to the strictly monotonicity and continuity of the map $r(\cdot,\tau)$ and the fact that $r^0 \in (0,r_T]$, we can find a number $M_0\in [0,M^\tau)$ and an optimal control $u^0\in\mathcal U_{M_0,\tau}$ to $(OP)^{M_0,\tau}$ such that $r^0 = r(M_0,\tau)$ and
 \begin{equation}\label{wang3.18}
  \|y(T;\chi_{(\tau,T)}u^0,y_0) - z_d\| = r(M_0,\tau) = r^0 \leq r.
 \end{equation}
From the above, we see that the control $u^0\in \mathcal W_{\tau,r}\bigcap\mathcal U_{M_0,\tau}$. This, along with the optimality of $M(\tau,r)$, yields  that
 \begin{equation*}
  M(\tau,r) \leq \|\chi_{(\tau,T)}u^0\|_\infty \leq M_0 < M^\tau \mbox{ for } r > 0.
 \end{equation*}

 {\it Step 5. It holds that $r = r(M(\tau,r),\tau),~ r\in (0,r_T]$.}

 Let $0 < r \leq r_T$. An optimal control  $u_1$  to $(NP)^{\tau,r}$ verifies that
 \begin{equation*}
  \|y(T;\chi_{(\tau,T)}u_1,y_0) - z_d\| \leq r  \mbox{ and } \|\chi_{(\tau,T)}u_1\|_\infty = M(\tau,r).
 \end{equation*}
 From these and the definition of $r(M(\tau,r),\tau)$, we can deduce that $r\geq r(M(\tau,r),\tau)$. Then we claim that $r = r(M(\tau,r),\tau)$. Otherwise, we would get that $r > r(M(\tau,r),\tau)$. Using the fact that $M(\tau,r)< M^\tau$ (see Step 4) and the fact that the map $r(\cdot,\tau)$ is strictly monotonically decreasing and continuous when $M\in[0,M^\tau)$, we can find a number $M_1\in [0,M^\tau)$ such that $r = r(M_1,\tau)$ with $M_1 < M(\tau,r)$. Then there is an optimal control $u_2$ to $(OP)^{M_1,\tau}$ such that
 \begin{equation*}
  \|y(T;\chi_{(\tau,T)}u_2,y_0) - z_d\| = r(M_1,\tau) =r  \mbox{ and } \|\chi_{(\tau,T)}u_2\|_\infty \leq M_1 < M(\tau,r).
 \end{equation*}
 This contradicts to the optimality of $M(\tau,r)$.

 {\it Step 6. $M=M(\tau,r(M,\tau)), ~ 0\leq M < M^\tau$.}

 Let $0\leq M < M^\tau$. Then it holds that  $r(M,\tau) > 0$ (see $(ii)$ in Lemma \ref{Lemma OP}) and $M(\tau,r(M,\tau))<M^\tau$ (see Step 4). Set $r=r(M,\tau)$.
 Since $r=r(M(\tau,r),\tau)$ with $r\in (0,r_T]$ (see Step 5), we have
 \begin{equation*}
  r(M,\tau) = r(M(\tau,r(M,\tau)),\tau)  \mbox{ for every } M\in[0,M^\tau).
 \end{equation*}
 This, combined with the strictly monotonicity of $r(\cdot,\tau)$, implies that $M = M(\tau,r(M,\tau))$ for every $M\in[0,M^\tau)$.
\end{proof}

\begin{Proposition}\label{OPNP}
 $(i)$ When $\tau\in [0,T)$ and $0\leq M < M_\tau$, the optimal control to $(OP)^{M,\tau}$ is also an optimal control to $(NP)^{\tau,r(M,\tau)}$; $(ii)$ When $\tau\in [0,T)$ and $r\in (0,r_T]$, the optimal control to $(NP)^{\tau,r}$ is also an optimal control to $(OP)^{M(\tau,r),\tau}$; $(iii)$ When $\tau\in [0,T)$ and $r\in (0,r_T]$, $(NP)^{\tau,r}$ holds the bang-bang property and has the unique optimal control.
\end{Proposition}

\begin{proof}
 $(i)$ The optimal control $u_1$  to $(OP)^{M,\tau}$ satisfies that  $\|y(T;\chi_{(\tau,T)}u_1,y_0) - z_d\| = r(M,\tau)$ and $\|\chi_{(\tau,T)}u_1\|_\infty \leq M$.
 From these and (\ref{3.14}),  $u_1$ is also an optimal control to $(NP)^{\tau,r(M,\tau)}$; $(ii)$ The optimal control  $u_2$  to $(NP)^{\tau,r}$ verifies that  $\|y(T;\chi_{(\tau,T)}u_2,y_0) - z_d\| \leq r$ and $\|\chi_{(\tau,T)}u_2\|_\infty = M(\tau,r)$. From these and (\ref{3.13}),  $u_2$ is also an optimal control to $(OP)^{M(\tau,r),\tau}$;
 $(iii)$ follows from  $(ii)$ and Lemma \ref{Lemma 3.3}.
\end{proof}

\bigskip
\subsection{Equivalence between the optimal time and norm problems}

$~~~$ The following lemma guarantees the existence of optimal control to $(TP)^{M,r}$.

\begin{Lemma}\label{TP}
 Let $r\in (0,r_T)$ and $M\in [M(0,r),\infty)$. Then $(TP)^{M,r}$ has optimal controls. Moreover, it holds that $\tau(M,r)<T$.
\end{Lemma}

\begin{proof}
By Lemma \ref{Np},   $(NP)^{0,r}$ has an optimal control $u^*\in\mathcal W_{0,r}$ with $\|u^*\|_{\infty}=M(0,r)\leq M$ such that $y(T;u^*,y_0)\in B(z_d,r)$. From these and the fact that $\mathcal V_{M(0,r),r}\subset \mathcal V_{M,r}$, it follows that $\mathcal V_{M,r}\neq \emptyset$.

 For each $u\in\mathcal V_{M,r}$, we claim that $\tau_{M,r}(u)$ can be reached. For the purpose, we take a sequence $\{t_n\}\subset [0,T)$, with $\displaystyle\lim_{n\rightarrow\infty}t_n=\tau_{M,r}(u)$, such that $y(T;\chi_{(t_n,T)}u,y_0)\in B(z_d,r)$.
  Clearly,
  $$\chi_{(t_n,T)}u\rightarrow \chi_{(\tau_{M,r}(u),T)}u\;\;\mbox{strongly in}\;\;L^{2}((0,T);L^2(\Omega)).
  $$
  This implies that
 $$y(T;\chi_{(t_n,T)}u,y_0)\rightarrow y(T;\chi_{(\tau_{M,r}(u),T)}u,y_0)\;\;\mbox{strongly  in}\;\;L^2(\Omega).
 $$
  This implies that $y(T;\chi_{(\tau_{M,r}(u),T)}u,y_0)$ $\in B(z_d,r)$.

 Next, by the definition of $\tau(M,r)$, there is a sequence $\{u_n\}\subset\mathcal V_{M,r}$ such that $$\tau_{M,r}(u_n)\rightarrow\tau(M,r).$$
  Since $\{u_n\}$ is bounded, we can find a subsequence $\{u_{n_k}\}\subset\{u_n\}$ and a control $\tilde u\in L^{\infty}((0,T);L^2(\Omega))$ such that
 \begin{equation*}
  \chi_{(\tau_{M,r}(u_{n_k}),T)}u_{n_k} \longrightarrow \chi_{(\tau(M,r),T)}\tilde u \mbox{ weakly star in } L^{\infty}((0,T);L^2(\Omega)).
 \end{equation*}
 Hence
 \begin{equation*}
  y(T;\chi_{(\tau_{M,r}(u_{n_k}),T)}u_{n_k},y_0) \rightarrow y(T;\chi_{(\tau(M,r),T)}\tilde u,y_0) \mbox{ weakly in } L^2(\Omega)
 \end{equation*}
 and
 \begin{equation*}
  \|\chi_{(\tau(M,r),T)}\tilde u\|_\infty \leq \displaystyle\liminf_{k\rightarrow\infty}\|\chi_{(\tau_{M,r}(u_{n_k}),T)}u_{n_k}\|\leq M.
 \end{equation*}
 So $y(T;\chi_{(\tau(M,r),T)}\tilde u,y_0)\in B(z_d,r)$ and $\tilde u\in\mathcal V_{M,r}$. Thus,  $\tilde u$ is an optimal control to $(TP)^{M,r}$.

 Finally, the inequality that
$\tau(M,r)<T$ follows from the fact that $r<r_T$.
\end{proof}

\begin{Lemma}\label{Lemma 3.8}
 Let $r\in(0,r_T)$. The map $M(\cdot,r):~ [0,T)\longrightarrow [M(0,r),\infty)$ is strictly monotonically increasing and continuous with $\displaystyle\lim_{\tau\rightarrow T} M(\tau,r)=\infty$. Moreover,
 \begin{equation}\label{3.15}
  M=M(\tau(M,r),r) \mbox{ for every } M\geq M(0,r),
 \end{equation}
 and
 \begin{equation}\label{3.16}
  \tau=\tau(M(\tau,r),r) \mbox{ for every } \tau\in[0,T).
 \end{equation}
 Consequently, the inverse of the map $M(\cdot,r)$ is the map $\tau(\cdot,r)$.
\end{Lemma}

\begin{proof}
We carry out the proof by several steps.

 \textit{Step 1. The map $M(\cdot,r)$ is strictly monotonically increasing.}

 Let $0\leq \tau_1 < \tau_2 < T$. We are going to show that $M(\tau_1,r)<M(\tau_2,r)$.
  Seeking for a contradiction, we could suppose that  $M(\tau_1,r)\geq M(\tau_2,r)$.
  Then the optimal control $u^*$ to $(NP)^{\tau_2,r}$ would satisfy that $$\|\chi_{(\tau_2,T)}u^*\|_{\infty}=M(\tau_2,r)\leq M(\tau_1,r).$$ Meanwhile, it is clear that  $$\chi_{(\tau_2,T)}u^*\in\mathcal W_{\tau_2,r}\subset\mathcal W_{\tau_1,r}.$$
   Now the above two results imply that $\chi_{(\tau_2,T)}u^*$ is also an optimal control to $(NP)^{\tau_1,r}$. Since  $\chi_{(\tau_2,T)}u^*(t)=0$ for a.e. $t\in (\tau_1,\tau_2)$ and because $\tau_1<\tau_2$, , we are led to a contradiction  to the bang-bang property of $(NP)^{\tau_1,r}$ (see $(iii)$ in Proposition \ref{OPNP}).

 \textit{Step 2. The map $M(\cdot,r)$ is left-continuous.}

 Let $0\leq \tau_1<\tau_2<\cdots<\tau_n<\cdots<\tau<T$ and $\displaystyle\lim_{n\rightarrow\infty} \tau_n =\tau$. It suffices to show that
 \begin{equation}\label{WANG3.19}
 \displaystyle\lim_{n\rightarrow\infty}M(\tau_n,r)=M(\tau,r).
 \end{equation}
  If (\ref{WANG3.19}) did not hold, then  by the strictly increasing monotonicity of the map $M(\cdot,r)$, we could have
  \begin{equation}\label{wang3.19}
  \displaystyle\lim_{n\rightarrow\infty}M(\tau_n,r)=M(\tau,r)-\delta\;\;
  \mbox{for some}\;\;\delta>0.
  \end{equation}
   The optimal control $u_n$ to $(NP)^{\tau_n,r}$ satisfies that
   \begin{equation}\label{wang3.20}
   \|\chi_{(\tau_n,T)}u_n\|_{\infty}=M(\tau_n,r)<M(\tau,r)\;\;\mbox{and}\;\;y(T;\chi_{(\tau_n,T)}u_n,y_0)\in B(z_d,r).
    \end{equation}
    Since $\{u_n\}$ is bounded and $\displaystyle\lim_{n\rightarrow\infty} \tau_n =\tau$, we can find a subsequence $\{u_{n_k}\}\subset\{u_n\}$ and a control $\tilde u\in L^\infty((0,T);L^2(\Omega))$ such that
   \begin{equation}\label{wang3.21}
   \chi_{(\tau_{n_k},T)}u_{n_k}\rightarrow\chi_{(\tau,T)}\tilde u\;\;\mbox{ weakly star in}\;\; L^\infty((0,T);L^2(\Omega)).
   \end{equation}
    This, along with  the first fact in (\ref{wang3.20}) and (\ref{wang3.19}), indicates  that
 \begin{equation}\label{zhangbb}
  \|\chi_{(\tau,T)}\tilde u\|_{\infty}\leq\displaystyle\liminf_{k\rightarrow\infty}\|\chi_{(\tau_{n_k},T)}u_{n_k}\|_{\infty}\leq\displaystyle\liminf_{k\rightarrow\infty} M(\tau_{n_k},r)\leq M(\tau,r)-\delta.
  \end{equation}
  Meanwhile, by (\ref{wang3.21}), one can easily show that
  $$y(T;\chi_{(\tau_{n_k},T)}u_{n_k},y_0)\rightarrow y(T;\chi_{(\tau,T)}\tilde u,y_0)\;\;\mbox{ weakly in}\;\; L^2(\Omega).
  $$
  This, together with the second fact in (\ref{wang3.20}), yields that
  \begin{equation}\label{wang3.23}
  \|y(T;\chi_{(\tau,T)}\tilde u,y_0)-z_d\|\leq\displaystyle\liminf_{k\rightarrow\infty} \|y(T;\chi_{(\tau_{n_k},T)}u_{n_k},y_0)-z_d\| \leq r.
 \end{equation}
 From (\ref{wang3.23}), we see that $\tilde u\in\mathcal W_{\tau,r}$. Then, by the optimality of $M(\tau,r)$, we must have that
 $$\|\chi_{(\tau,T)}\tilde u\|_\infty\geq M(\tau,r),
 $$
 which contradicts to
  (\ref{zhangbb}). Hence, (\ref{WANG3.19}) holds.

 \textit{Step 3. The map $M(\cdot,r)$ is right-continuous.}

 Let $0\leq\tau<\cdots<\tau_n<\cdots<\tau_2<\tau_1<T$ and $\displaystyle\lim_{n\rightarrow\infty} \tau_n =\tau$.
 It suffices to show that
 \begin{equation}\label{wang3.25}
 \displaystyle\lim_{n\rightarrow\infty}M(\tau_n,r)=M(\tau,r).
  \end{equation}
   Seeking for a contradiction, we suppose that (\ref{wang3.25}) did not hold. Then by the strictly increasing monotonically of the map $M(\cdot,r)$, we would have that
   \begin{equation}\label{wang3.26}
   \displaystyle\lim_{n\rightarrow\infty}M(\tau_n,r)=M(\tau,r)+\delta\;\;\mbox{ for some}\;\;\delta>0.
    \end{equation}
  Let $u_n$ and $y_n$ be the optimal control and the optimal state to $(NP)^{\tau_n,r}$ with $n\in\mathbb N^+$. Then by Proposition \ref{OPNP} and the optimality of $(OP)^{M(\tau_n,r),\tau_n}$,  we see that for all $u\in\mathcal U_{M(\tau_n,r),\tau_n}$,
 \begin{equation}\label{3.17}
  0<r=r(M(\tau_n,r),\tau_n)=\|y(T;\chi_{(\tau_n,T)}u_n,y_0)-z_d\| \leq \|y(T;\chi_{(\tau_n,T)}u,y_0)-z_d\|.
 \end{equation}
 Because $\|\chi_{(\tau_n,T)}u_n\|_\infty=M(\tau_n,r)\leq M(\tau_1,r)$,  we can find a subsequence $\{u_{n_k}\}\subset\{u_n\}$ and a control  $\tilde u\in L^\infty((0,T);L^2(\Omega))$ such that
 \begin{equation*}
  \chi_{(\tau_{n_k},T)}u_{n_k} \longrightarrow \tilde u \mbox{ weakly star in } L^\infty((0,T);L^2(\Omega)).
 \end{equation*}
 This implies that
 \begin{equation}\label{3.18}
  y_{n_k}(T)\rightarrow y(T;\chi_{(\tau,T)}\tilde u,y_0) \mbox{ weakly in }L^2(\Omega)
 \end{equation}
 and
 \begin{equation}\label{3.19}
  \|\chi_{(\tau,T)}\tilde u\|_\infty \leq \displaystyle\liminf_{k\rightarrow\infty} \|\chi_{(\tau_{n_k},T)}u_{n_k}\|_\infty \leq \displaystyle\liminf_{k\rightarrow\infty} M(\tau_{n_k},r)=M(\tau,r)+\delta.
 \end{equation}
 Here, we used (\ref{wang3.26}).
 Meanwhile,  it is clear that
 \begin{equation}\label{wang3.30}
 \frac{M(\tau_{n_k},r)}{M(\tau,r)+\delta}\chi_{(\tau_{n_k},T)}v\in\mathcal U_{M(\tau_{n_k},r),\tau_{n_k}}\;\;\mbox{for each}\;\; v\in\mathcal U_{M(\tau,r)+\delta,\tau}.
 \end{equation}
It follows from (\ref{wang3.30}) and (\ref{3.17}) that
 \begin{equation*}
  0< ~r ~\leq ~\Big\|y(T;\frac{M(\tau_{n_k},r)}{M(\tau,r)+\delta}\chi_{(\tau_{n_k},T)}v,y_0)-z_d\Big\|\;\;\mbox{for all}\;\; v\in\mathcal U_{M(\tau,r)+\delta,\tau}\;\;\mbox{and all}\;\; k\in \mathbb{N}.
 \end{equation*}
Sending  $k\rightarrow\infty$ in the above, we get
 \begin{equation}\label{3.20}
  r \leq \|y(T;\chi_{(\tau,T)}v,y_0)-z_d\| \mbox{ for all } v\in\mathcal U_{M(\tau,r)+\delta,\tau}.
 \end{equation}

On the other hand, it follows from  (\ref{3.18}) that
 \begin{equation*}
  \|y(T;\chi_{(\tau,T)}\tilde u,y_0)-z_d\| \leq \displaystyle\liminf_{k\rightarrow\infty}\|y_{n_k}(T)-z_d\|=r>0.
 \end{equation*}
 This, along with (\ref{3.19}) and (\ref{3.20}), indicates that
 \begin{equation}\label{3.21}
  0<r=\|y(T;\chi_{(\tau,T)}\tilde u,y_0)-z_d\| \leq \|y(T;\chi_{(\tau,T)}v,y_0)-z_d\|  \mbox{ for all } v\in\mathcal U_{M(\tau,r)+\delta,\tau}.
 \end{equation}

 By  (\ref{3.21}) and (\ref{3.19}), we see that $\tilde u$ is an optimal control to $(OP)^{M(\tau,r)+\delta,\tau}$ and
 \begin{equation}\label{wang3.33}
 r(M(\tau,r)+\delta,\tau)=r >0.
 \end{equation}
Clearly, it holds that
\begin{equation}\label{wang3.34}
M(\tau,r)+\delta < M^\tau,
\end{equation}
 for otherwise we could have that  $r(M(\tau,r)+\delta,\tau)=0$, which contradicts to (\ref{wang3.33}).

 Now, (\ref{wang3.34}) and the facts  that $r(M(\tau,r)+\delta,\tau)=r$ and $r(M(\tau,r),\tau)=r$ (see (\ref{3.13})), leads to a contradiction to  the strictly monotonicity of $r(\cdot,\tau)$.
 Hence, (\ref{wang3.25}) stands.

 \textit{Step 4. $\displaystyle\lim_{\tau\rightarrow T} M(\tau,r)=\infty$, when $r\in(0,r_T)$.}

 If it didn't hold, we could have that $\displaystyle\lim_{\tau\rightarrow T} M(\tau,r) \leq N_0$ for some $N_0 > 0$. We set $\{\tau_n\}\subset [0,T)$ and $\displaystyle\lim_{n\rightarrow\infty} \tau_n =T$. Let $u_n$ be the optimal control to $(NP)^{\tau_n,r}$ for $n\in\mathbb N^+$. Since $\{u_n\}$ is bounded in $L^\infty((0,T);L^2(\Omega))$,
 there is a subsequence of $\{u_n\}$, still denoted in the same way, such that
 \begin{equation*}
 \chi_{(\tau_n,T)}u_n\rightarrow 0\;\;\mbox{ weakly star in}\;\;L^\infty((0,T);L^2(\Omega)),
 \end{equation*}
  from which, it follows that
  $$y(T;\chi_{(\tau_n,T)}u_n,y_0)\rightarrow y(T;0,y_0)\;\;\mbox{ weakly in}\;\; L^2(\Omega). $$
  Hence, it holds that
 \begin{equation*}
  r_T=\|y(T;0,y_0)-z_d\| \leq \displaystyle\liminf_{n\rightarrow\infty} \|y(T;\chi_{(\tau_n,T)}u_n,y_0)-z_d\| \leq r.
 \end{equation*}
 This contradicts to the fact that $r < r_T$.

 \textit{Step 5. $M=M(\tau(M,r),r)$ for every $M\geq M(0,r)$.}

 Let $M\geq M(0,r)$. According to Lemma \ref{TP}, there is an optimal control $u_1\in\mathcal V_{M,r}$ to $(TP)^{M,r}$ such that
 \begin{equation*}
  y(T;\chi_{(\tau(M,r),T)}u_1,y_0) \in B(z_d,r) \mbox{ and } \|\chi_{(\tau(M,r),T)}u_1\|_\infty \leq M.
 \end{equation*}
 Hence, by the optimality of $M(\tau(M,r),r)$, we have $M(\tau(M,r),r)\leq M$. We claim that $M(\tau(M,r),r)=M$. Otherwise we could suppose that $M(\tau(M,r),r) < M$. Then according to the strictly increasing monotonicity and continuity of the map $M(\cdot,r)$, we can find a number $\tau_1\in (\tau(M,r),T)$ such that $M(\tau_1,r)=M$. Clearly, the optimal control  $u_2\in\mathcal W_{\tau_1,r}$  to $(NP)^{\tau_1,r}$ satisfies that
 \begin{equation*}
  y(T;\chi_{(\tau_1,T)}u_2,y_0)\in B(z_d,r) \mbox{ and } \|\chi_{(\tau_1,T)}u_2\|_\infty = M(\tau_1,r)=M.
 \end{equation*}
 Thus, it holds that  $u_2\in\mathcal V_{M,r}$. This, combined with the fact that $\tau_1 > \tau(M,r)$, contradicts to the optimality of $\tau(M,r)$.

\textit{ Step 6. $\tau=\tau(M(\tau,r),r)$ for every $\tau\in [0,T)$.}

 According to the increasing monotonicity of the map $M(\cdot,r)$ and Lemma \ref{TP}, we know that $M(\tau,r)\geq M(0,r)$ and $\tau(M(\tau,r),r)< T$.
 Since
 $$
 M=M(\tau(M,r),r),\;\;\mbox{when}\;\; M\geq M(0,r) \; (\mbox{see Step 5}),
 $$
 by taking
   $M=M(\tau,r)$ in the above, we obtain that $M(\tau,r)=M(\tau(M(\tau,r),r),r)$. From this, the strictly monotonicity of the map $M(\cdot,r)$ and the fact that $\tau(M(\tau,r),r)< T$, we have $\tau=\tau(M(\tau,r),r)$. This completes the proof.

\end{proof}

\begin{Remark}\label{Rk 3.9}
 $(i)$ Paying attention to Lemma \ref{TP} where $r\in(0,r_T)$, we see that $M\geq M(0,r)$  if an only if  $(TP)^{M,r}$ has optimal controls. The sufficiency has been proved in Lemma \ref{TP}. Now we show the necessity. Let $u^*$ be an optimal control to  $(TP)^{M,r}$. Then it holds that
 \begin{equation*}
  y(T;\chi_{(\tau(M,r),T)}u^*,y_0)\in B(z_d,r) \mbox{ and } \|\chi_{(\tau(M,r),T)}u^*\|_\infty \leq M.
 \end{equation*}
 From these and the optimality of $M(\tau(M,r),r)$, we have $M(\tau(M,r),r)\leq M$. This, together with that $M(\tau,r)$ is non-decreasing, indicates that $M(0,r)\leq M(\tau(M,r),r)\leq M$.

  $(ii)$ About the function $M^\tau$ (see (\ref{bb})), we have that $M<M^{\tau(M,r)}$, when $r\in (0,r_T)$.
  In fact, according to Lemma \ref{Lemma 3.8} and the fact that $M\geq M(0,r)$, one has that $M=M(\tau_0,r)$ for some $\tau_0\in[0,T)$. By this and Lemma \ref{Np},  we can find an optimal control $u^*$ to $(NP)^{\tau_0,r}$ with $\|\chi_{(\tau_0,T)}u^*\|_\infty=M(\tau_0,r)=M$ such that $y(T;\chi_{(\tau_0,T)}u^*,y_0)\in B(z_d,r)$. Then we have $u^*\in\mathcal V_{M,r}$ and hence $\tau_0 \leq \tau(M,r)$. By Lemma \ref{Lemma 3.5} and Lemma \ref{Lemma Mtau}, we conclude that $M=M(\tau_0,r) < M^{\tau_0} \leq M^{\tau(M,r)}$.

 $(iii)$ Let $A=\{(M,r); M >0,r\in[r(M,0),r_T)\cap(0,r_T)\}$ and $B=\{(M,r); M\geq M(0,r),r\in(0,r_T)\}$. Then it holds  that $A=B$. Indeed, if $(M,r)\in A$, then by the fact that $r\geq r(M,0)$ and the fact that $M(\tau,\cdot)$ is non-increasing, we see that $M(0,r)\leq M(0,r(M,0))$. Using $(i)$ in Remark \ref{Rk maximal}, we can find an optimal control $u^*$ to $(OP)^{M,0}$ such that $\|y(T;u^*,y_0)-z_d\|=r(M,0)$ and $\|u^*\|_\infty \leq M.$ From these and the optimality of $M(0,r(M,0))$, we see that $M(0,r(M,0))\leq M$. This, combined with the fact that $M(0,r)\leq M(0,r(M,0))$, indicates that $M\geq M(0,r)$. Hence $A\subset B$. Conversely, if $(M,r)\in B$, then by the fact that that $r\in(0,r_T)$,  Lemma \ref{Np} and (\ref{3.13}), it holds that $M(0,r)>0$ and $r=r(M(0,r),0)$. These, together with the facts that $M\geq M(0,r)$ and $r(\cdot,\tau)$ is non-increasing, indicate that $M\geq M(0,r)>0$ and that $r=r(M(0,r),0)\geq r(M,0)$. From these, it follows that  $B\subset A$.

\end{Remark}

\begin{Proposition}\label{TPNP}
 $(i)$ When $\tau\in[0,T)$ and $r\in(0,r_T)$, the optimal control to $(NP)^{\tau,r}$ is also an optimal control to $(TP)^{M(\tau,r),r}$;
 $(ii)$ When $(M,r)\in \{(M,r); M >0,r\in[r(M,0),r_T)\cap(0,r_T)\}$ or $(M,r)\in\{(M,r); M\geq M(0,r),r\in(0,r_T)\}$, the optimal control to $(TP)^{M,r}$ is also an optimal control to $(NP)^{\tau(M,r),r}$. Furthermore, $(TP)^{M,r}$ holds bang-bang property and has the unique optimal control.
\end{Proposition}

\begin{proof}
 $(i)$ The optimal control $u_1$ to $(NP)^{\tau,r}$ satisfies that $y(T;\chi_{(\tau,T)}u_1,y_0)\in B(z_d,r)$ and that $\|\chi_{(\tau,T)}u_1\|_\infty = M(\tau,r)$. From these and (\ref{3.16}), we can deduce that $u_1$ is also an optimal control to $(TP)^{M(\tau,r),r}$;
 $(ii)$ According to $(iii)$ in Remark \ref{Rk 3.9}, an optimal control $u_2$ to $(TP)^{M,r}$ verifies that $y(T;\chi_{(\tau(M,r),T)}u_2,y_0)\in B(z_d,r)$ and that $\|\chi_{(\tau(M,r),T)}u_2\|_\infty \leq M$. By these and (\ref{3.15}), we see that $u_2$ is also an optimal control to $(NP)^{\tau(M,r),r}$. Finally, from $(ii)$ and Proposition \ref{OPNP}, one can easily obtain the bang-bang property and the uniqueness for $(TP)^{M,r}$.
\end{proof}

\subsection{Equivalence between three optimal control problems}

\begin{proof}[Proof of Theorem \ref{Theorem 1.3}]
 We start with introducing two notations: Write $P_1=P_2$ for the statement that problems $P_1$ and $P_2$ have the  same optimal control; while $P_1\Rightarrow P_2$ for the statement that the optimal control to $P_1$ is also an optimal control to $P_2$. Hence,  $P_1=P_2$ if and only if $P_1\Rightarrow P_2$ and $P_2\Rightarrow P_1$.

We now organize the proof by three steps.

 \textit{Step 1. When $\tau\in[0,T)$ and $0< M <M^\tau$, $(OP)^{M,\tau} = (TP)^{M,r(M,\tau)} = (NP)^{\tau,r(M,\tau)}$.}

 By the definition of $\mathcal U_{M,0}$, every control $u\in\mathcal U_{M,\tau}$ satisfies that $\chi_{(\tau,T)}u\in\mathcal U_{M,0}$. Then we easily see that $r(M,\tau)\geq r(M,0)$. This, combined with the assumption, $(ii)$ in Lemma \ref{Lemma OP} and Lemma \ref{Lemma 3.5}, implies that
 \begin{equation}\label{1}
  \tau\in[0,T), ~ M\in(0,M^\tau), ~ r(M,\tau)\in[r(M,0),r_T)\cap (0,r_T)
 \end{equation}
 and
 \begin{equation}\label{2}
  M=M(\tau,r(M,\tau)) , ~ \tau=\tau(M(\tau,r(M,\tau)),r(M,\tau))=\tau(M,r(M,\tau)).
 \end{equation}
 Using Proposition \ref{OPNP} and (\ref{1}), we have
  $$
  (OP)^{M,\tau}\Rightarrow (NP)^{\tau,r(M,\tau)}\;\;\mbox{ and}\;\;(NP)^{\tau,r(M,\tau)}\Rightarrow (OP)^{M(\tau,r(M,\tau)),\tau}.
   $$
   Then it follows from (\ref{2}) that $(OP)^{M,\tau} = (NP)^{\tau,r(M,\tau)}$. On the other hand, according to Proposition \ref{TPNP}, (\ref{1}) and (\ref{2}), we obtain that
 $$(NP)^{\tau,r(M,\tau)}\Rightarrow (TP)^{M(\tau,r(M,\tau)),r(M,\tau)}$$
  and
  $$
  (TP)^{M(\tau,r(M,\tau)),r(M,\tau)}\Rightarrow (NP)^{\tau(M(\tau,r(M,\tau)),r(M,\tau)),r(M,\tau)}.
  $$
  These, together with (\ref{2}) again, indicate that $(NP)^{\tau,r(M,\tau)}=(TP)^{M,r(M,\tau)}$.

 \textit{Step 2. When $\tau\in[0,T)$ and $r\in(0,r_T)$,  $(NP)^{\tau,r} = (OP)^{M(\tau,r),\tau} = (TP)^{M(\tau,r),r}$.}

 According to the assumption, Lemma \ref{Lemma 3.8}, Lemma \ref{Np} and Lemma \ref{Lemma 3.5}, we can conclude that
 \begin{equation}\label{3}
  \tau\in[0,T), ~ r\in (0,r_T), ~ M(\tau,r)\in[M(0,r),M^\tau)\cap (0,M^\tau)
 \end{equation}
 and
 \begin{equation}\label{4}
  r=r(M(\tau,r),\tau), ~ \tau=\tau(M(\tau,r),r).
 \end{equation}
By Proposition \ref{OPNP} and (\ref{3}), we have that
 $$(NP)^{\tau,r}\Rightarrow (OP)^{M(\tau,r),\tau}\;\;\mbox{ and}\;\;(OP)^{M(\tau,r),\tau}\Rightarrow (NP)^{\tau,r(M(\tau,r),\tau)}.$$
  From these and (\ref{4}) one can deduce that $(NP)^{\tau,r}=(OP)^{M(\tau,r),\tau}$. On the other hand, by Proposition \ref{TPNP} and (\ref{3}), it holds that
  $$
  (NP)^{\tau,r}\Rightarrow (TP)^{M(\tau,r),r}\;\;\mbox{ and}\;\;(TP)^{M(\tau,r),r}\Rightarrow (NP)^{\tau(M(\tau,r),r),r}.
   $$
   These, along with (\ref{4}), lead to that $(NP)^{\tau,r} =(TP)^{M(\tau,r),r}$.

 \textit{Step 3. When $M >0$ and $r\in [r(M,0),r_T)\cap(0,r_T)$,  $(TP)^{M,r} = (NP)^{\tau{(M,r)},r} = (OP)^{M,\tau{(M,r)}}$.}

 From the assumption, Lemma \ref{TP} and $(ii)$ in Remark \ref{Rk 3.9}, we  have
 \begin{equation}\label{5}
  \tau(M,r)\in[0,T), ~ r\in [r(M,0),r_T)\cap(0,r_T), ~ 0<M<M^{\tau(M,r)}
 \end{equation}
 and
 \begin{equation}\label{6}
  M=M(\tau(M,r),r), ~ r=r(M(\tau(M,r),r),\tau(M,r))=r(M,\tau(M,r)).
 \end{equation}
 By Proposition \ref{TPNP} and (\ref{5}), we obtain that
  $$(TP)^{M,r}\Rightarrow (NP)^{\tau{(M,r)},r}\;\;\mbox{ and}\;\;(NP)^{\tau{(M,r)},r}\Rightarrow (TP)^{M(\tau(M,r),r),r}.
   $$
   From these and (\ref{6}), it holds that $(TP)^{M,r}=(NP)^{\tau{(M,r)},r}$. On the other hand, according to Proposition \ref{OPNP}, (\ref{5}) and (\ref{6}), we deduce that $$(NP)^{\tau{(M,r)},r}\Rightarrow (OP)^{M(\tau(M,r),r),\tau{(M,r)}}
   $$
    and
    $$
    (OP)^{M(\tau(M,r),r),\tau{(M,r)}}\Rightarrow (NP)^{\tau{(M,r)},r(M(\tau(M,r),r),\tau(M,r))}.$$
    From these and (\ref{6}), one can obtain that $(NP)^{\tau{(M,r)},r} = (OP)^{M,\tau{(M,r)}}$.
\end{proof}

\noindent\textbf{Acknowledgement.} \;The author appreciates Professor Gengsheng Wang for his valuable help in this paper.

\end{document}